\documentclass[10pt]{amsart}
\usepackage{verbatim}
\usepackage{eucal,url,amssymb,stmaryrd,enumerate,amscd,verbatim}
\usepackage{amsmath}
\usepackage[pagebackref,colorlinks=true,linkcolor=blue,citecolor=blue]{hyperref}
\usepackage{amsfonts}
\usepackage{amsthm}
\usepackage[margin=1in]{geometry}

\allowdisplaybreaks
\numberwithin{equation}{section}
\newtheorem{thrm}{Theorem}[section]
\newtheorem{lemma}[thrm]{Lemma}
\newtheorem{prop}[thrm]{Proposition}
\newtheorem{cor}[thrm]{Corollary}
\newtheorem{dfn}[thrm]{Definition}

\newtheorem{rmrk}[thrm]{Remark}

\newtheorem{conv}[thrm]{Convention}

 1

\def\gr{\nabla f}
\def\g{\nabla\varphi}
\def\bi{\nabla}

\newcommand{\vol}{\, Vol_{\eta}}
\newcommand{\Vol}{\, Vol_{\eta}}

\begin{document}

\begin{abstract}
We establish quaternionic contact (qc) versions of the so called Almost Schur Lemma, which give estimations of the qc scalar curvature on a compact qc manifold to be a constant in terms of the norm of the $[-1]$-component and the norm of the trace-free part of the $[3]$-component of the horizontal qc Ricci tensor and the torsion endomorphism, under certain positivity conditions. 
\end{abstract}

\keywords{Quaternionic contact structure, QC Lichnerowicz condition, QC Cordes estimate, P-function}
\subjclass[2010]{53C21, 58J60, 53C17, 35P15, 53C25}
\title[The Almost Schur Lemma in Quaternionic Contact Geometry]{The Almost Schur Lemma in Quaternionic Contact Geometry}
\date{\today }
\author{Stefan Ivanov}
\address[Stefan Ivanov]{University of Sofia, Faculty of Mathematics and
Informatics, blvd. James Bourchier 5, 1164, Sofia, Bulgaria}
\address{and Institute of Mathematics and Informatics, Bulgarian Academy of
Sciences} \email{ivanovsp@fmi.uni-sofia.bg}
\author{Alexander Petkov}
\address[Alexander Petkov]{University of Sofia, Faculty of Mathematics and
Informatics, blvd. James Bourchier 5, 1164, Sofia, Bulgaria}
\email{a\_petkov\_fmi@abv.bg}

\maketitle

\tableofcontents


\setcounter{tocdepth}{2}

\section{Introduction}
For a compact Riemannian manifold $(M^n,g)$ of dimension $n\geq 3$ the famous Schur lemma states that if $(M^n,g)$ is Einstein, then it has constant scalar curvature, $S=Const.$ The metric $g$ is said to be Einstein, if the Ricci tensor is proportional to the metric, $Ric=\frac{S}{n}g.$ A generalization of the Schur lemma is a recent result of De Lellis and Topping \cite{LT} that states as follows.
\begin{thrm}\cite[Almost Schur Lemma]{LT} Let $(M^n,g)$ be a compact Riemannian manifold of dimension $n\geq 3$ with non--negative Ricci tensor, $Ric\geq 0$. Then the following inequality holds
\begin{equation*}\label{ASL}
\int_M(S-\bar{S})^2 vol_g\leq\frac{4n(n-1)}{(n-2)^2}\int_M\left|Ric-\frac{S}{n}g\right|_g^2 vol_g,
\end{equation*}
where $\bar{S}$ means the average value of the scalar curvature $S$ of $g$.

The equality holds if and only if the manifold is Einstein.
\end{thrm}
It is also shown in \cite{LT} that the positivity condition $Ric\geq 0$ assumed on the Ricci tensor is essential and can not be dropped. 

In the CR case there are known two positivity conditions written in terms of the Webster Ricci curvature and the pseudohermitian torsion; one is used for obtaining a lower bound of the first egenvalue of the sub-Laplacian (see e.g. \cite{Gr}), while the other one appears in the CR  Cordes type estimate (see \cite{CM}). A CR  version of the almost Schur lemma was first established in \cite{CSW} under the Greenleaf's positivity condition. Recently, a CR version of the almost Schur Lemma was presented in \cite{IP0} under the CR Cordes positivity condition which gives better estimate in the torsion-free (Sasakian)  case.

The aim of this note  is to present  a quaternionic contact (qc) version of the Almost Schur lemma. Quaternionic contact geometry is an example of sub--Riemannian geometry and any quaternionic contact manifold  admits a canonical connection $\nabla$, namely the Biquard connection, whose role in qc geometry is similar to that of Levi--Civita connection in Riemannian geometry and Tanaka--Webster connection in CR geometry. We present in terms of the Biquard connection  two versions of quaternionic contact  almost Schur lemma depending on the positivity assumption,  see Theorem~\ref{thrm1} and Theorem~\ref{thrm2} below.


\begin{conv}
\label{conven} \hfill\break\vspace{-15pt}

\begin{enumerate}[ a)]

\item We shall use $X,Y,Z,U$ to denote horizontal vector fields, i.e. $%
X,Y,Z,U\in H$.

\item $\{e_1,\dots,e_{4n}\}$ denotes a local orthonormal basis of the
horizontal space $H$.


\item The triple $(i,j,k)$ denotes any cyclic permutation of
$(1,2,3)$.

\item $s$ will be any number from the set $\{1,2,3\}$,
$s\in\{1,2,3\} $.
\end{enumerate}
\end{conv}

\textbf{Acknowledgments}  The research of both authors  is partially supported   by Contract DH/12/3/12.\allowbreak{}12.2017,
Contract 80-10-161/05.04.2021  with the Sofia University ``St. Kliment Ohridski'', 
and
the National Science Fund of Bulgaria, National Scientific Program ``VIHREN'', Project No. KP-06-DV-7.

\section{Quaternionic contact manifolds}

It is well known that the sphere at infinity of a non-compact
symmetric space of rank one carries a natural
Carnot-Carath\'eodory structure, see \cite{M,P}. A quaternionic
contact (qc) structure, \cite{Biq1}, appears naturally as the
conformal boundary at infinity of the quaternionic hyperbolic
space.

Quaternionic contact manifolds were introduced in \cite{Biq1}. We also refer
to \cite{IMV} and \cite{IV} for further results and background.

\subsection{Quaternionic contact structures and the Biquard connection}

\label{ss:Biq conn}

A quaternionic contact (qc) manifold $(M, g, \mathbb{Q})$ is a $4n+3$%
-dimensional manifold $M$ with a codimension three distribution $H$ equipped
with an $Sp(n)Sp(1)$-structure. Explicitly, $H$ is the kernel of a local
1-form $\eta=(\eta_1,\eta_2,\eta_3)$ with values in $\mathbb{R}^3$ together
with a compatible Riemannian metric $g$ and a rank-three bundle $\mathbb{Q}$
consisting of endomorphisms of $H$, locally generated by three almost complex
structures $I_1,I_2,I_3$ on $H$, satisfying the identities of the imaginary
unit quaternions. Thus, we have $I_1I_2=-I_2I_1=I_3, \quad
I_1I_2I_3=-id_{|_H}$, which are hermitian compatible with the metric $%
g(I_s.,I_s.)=g(.,.)$, and the following compatibility conditions
hold $2g(I_sX,Y)\ =\ d\eta_s(X,Y)$.

On a qc manifold of dimension $(4n+3)>7$ with a fixed metric $g$
on  $H$ there exists a canonical connection defined in
\cite{Biq1}. Biquard also showed that there is a unique connection $%
\nabla$ with torsion $T$ and a unique supplementary subspace $V$ to $H$ in $%
TM$, such that:
\begin{enumerate}[(i)]
\item $\nabla$ preserves the splitting $H\oplus V$ and the
$Sp(n)Sp(1)$-structure on $H$, i.e., $\nabla g=0, \nabla\sigma
\in\Gamma( \mathbb{Q})$ for a section
$\sigma\in\Gamma(\mathbb{Q})$, and its torsion on $H$ is given by
$T(X,Y)=-[X,Y]_{|V}$;
\item for $\xi\in V$, the endomorphism $T(\xi,.)_{|H}$ of $H$ lies in $%
(sp(n)\oplus sp(1))^{\bot}\subset gl(4n)$; \item the connection on
$V$ is induced by the natural identification $\varphi
$ of $V$ with $\mathbb Q$, 
$\nabla\varphi=0$.
\end{enumerate}


If the dimension of $M$ is at least eleven \cite{Biq1} also
described the supplementary \emph{vertical distribution} $V$,
which is (locally) generated by the so called \emph{Reeb vector
fields} $\{\xi_1,\xi_2,\xi_3\}$ determined by
\begin{equation}  \label{bi1}
\begin{aligned}
\eta_s(\xi_k)=\delta_{sk}, \qquad (\xi_s\lrcorner d\eta_s)_{|H}=0,
\qquad (\xi_s\lrcorner d\eta_k)_{|H}=-(\xi_k\lrcorner
d\eta_s)_{|H},
\end{aligned}
\end{equation}
where $\lrcorner$ denotes the interior multiplication.

If the dimension of $M $ is seven Duchemin shows in \cite{D} that if we
assume, in addition, the existence of Reeb vector fields as in \eqref{bi1},
then the Biquard result holds. \emph{Henceforth, by a qc structure in
dimension $7$ we shall mean a qc structure satisfying \eqref{bi1}}. This
implies the existence of the connection with properties (i), (ii) and (iii)
above.

The fundamental 2-forms $\omega_s$ of the quaternionic structure  are
defined by
\begin{equation*}
2\omega_{s|H}\ =\ \, d\eta_{s|H},\qquad \xi\lrcorner\omega_s=0,\quad \xi\in
V.
\end{equation*}
The torsion restricted to $H$ has the form $
T(X,Y)=-[X,Y]_{|V}=2\sum_{s=1}^3\omega_s(X,Y)\xi_s.
$

\subsection{Invariant decompositions}

Any endomorphism $\Psi$ of $H$ can be decomposed with respect to the
quaternionic structure $(\mathbb{Q},g)$ uniquely into four $Sp(n)$-invariant
parts
$\Psi=\Psi^{+++}+\Psi^{+--}+\Psi^{-+-}+\Psi^{--+},$ 
where $\Psi^{+++}$ commutes with all three $I_i$, $\Psi^{+--}$ commutes with
$I_1$ and anti-commutes with the others two, etc. The two $Sp(n)Sp(1)$-invariant components are given by $$\Psi_{[3]}=\Psi^{+++},\quad
\Psi_{[-1]}=\Psi^{+--}+\Psi^{-+-}+\Psi^{--+}$$ with the following characterizing equations
\begin{eqnarray*} 
\Psi&=&\Psi_{[3]}\iff 3\Psi+I_1\Psi I_1+I_2\Psi I_2+I_3\Psi I_3=0,\\
\Psi&=&\Psi_{[-1]}\iff \Psi-I_1\Psi I_1-I_2\Psi I_2-I_3\Psi I_3=0.
\end{eqnarray*}
These are the projections on the eigenspaces of the Casimir
operator
$$\Upsilon =\ I_1\otimes I_1\ +\ I_2\otimes I_2\ +\ I_3\otimes I_3,$$
corresponding, respectively, to the eigenvalues $3$ and $-1$, see \cite{CSal}%
. Note here that each of the three 2-forms $\omega_s$ belongs to the
[-1]-component, $\omega_s=\omega_{s[-1]}$, and constitute a basis of the Lie
algebra $sp(1)$.

If $n=1$ then the space of symmetric endomorphisms commuting with all $I_s$
is 1-dimensional, i.e., the [3]-component of any symmetric endomorphism $\Psi
$ on $H$ is proportional to the identity, $\Psi_{[3]}=-%
\frac{tr\Psi}{4}Id_{|H}$.

\subsection{The torsion tensor}

The torsion endomorphism $T_{\xi }=T(\xi ,\cdot ):H\rightarrow H,\quad \xi
\in V$ will be decomposed into its symmetric part $T_{\xi }^{0}$ and
skew-symmetric part $b_{\xi }, T_{\xi }=T_{\xi }^{0}+b_{\xi }$. Biquard
showed in \cite{Biq1} that the torsion $T_{\xi }$ is completely trace-free, $%
tr\,T_{\xi }=tr\,(T_{\xi }\circ I_{s})=0$, its symmetric part has the
properties $T_{\xi _{i}}^{0}I_{i}=-I_{i}T_{\xi _{i}}^{0},\quad I_{2}(T_{\xi
_{2}}^{0})^{+--}=I_{1}(T_{\xi _{1}}^{0})^{-+-},\quad I_{3}(T_{\xi
_{3}}^{0})^{-+-}=I_{2}(T_{\xi _{2}}^{0})^{--+},\quad I_{1}(T_{\xi
_{1}}^{0})^{--+}=I_{3}(T_{\xi _{3}}^{0})^{+--}$.
The skew-symmetric part can be represented as $b_{\xi _{i}}=I_{i}U$, where $U
$ is a traceless symmetric (1,1)-tensor on $H$ which commutes with $%
I_{1},I_{2},I_{3}$. Therefore we have $T_{\xi _{i}}=T_{\xi _{i}}^{0}+I_{i}U$%
. When $n=1$ the tensor $U$ vanishes identically, $U=0$, and the
torsion is a symmetric tensor, $T_{\xi }=T_{\xi }^{0}.$

The two $Sp(n)Sp(1)$-invariant trace-free symmetric 2-tensors on $H$
\begin{equation} \label{Tcompnts}
T^0(X,Y)= g((T_{\xi_1}^{0}I_1+T_{\xi_2}^{0}I_2+T_{ \xi_3}^{0}I_3)X,Y) \
\text{ and }\ U(X,Y) =g(UX,Y)
\end{equation}
were introduced in \cite{IMV} and enjoy the properties
\begin{equation}  \label{propt}
\begin{aligned} T^0(X,Y)+T^0(I_1X,I_1Y)+T^0(I_2X,I_2Y)+T^0(I_3X,I_3Y)=0, \\
U(X,Y)=U(I_1X,I_1Y)=U(I_2X,I_2Y)=U(I_3X,I_3Y). \end{aligned}
\end{equation}
From \cite[Proposition~2.3]{IV} we have
\begin{equation}  \label{need}
4T^0(\xi_s,I_sX,Y)=T^0(X,Y)-T^0(I_sX,I_sY),
\end{equation}
hence, taking into account \eqref{need} it follows
\begin{multline}  \label{need1}
T(\xi_s,I_sX,Y)=T^0(\xi_s,I_sX,Y)+g(I_sUI_sX,Y) 
=\frac14\Big[T^0(X,Y)-T^0(I_sX,I_sY)\Big]-U(X,Y).
\end{multline}
\subsection{Torsion and curvature}Let $R=[\nabla,\nabla]-\nabla_{[\ ,\ ]}$ be the curvature tensor of $\nabla$
and the dimension is $4n+3$. We denote the curvature tensor of type (0,4)
and the torsion tensor of type (0,3) by the same letter, $%
R(A,B,C,D):=g(R(A,B)C,D),\quad T(A,B,C):=g(T(A,B),C)$, $A,B,C,D
\in \Gamma(TM)$. The  Ricci tensor, the normalized
scalar curvature and the  Ricci $2$-forms  of the Biquard connection, called \emph{qc-Ricci tensor} $Ric$,
\emph{normalized qc-scalar curvature} $S$ and \emph{qc-Ricci forms} $\rho_s$, respectively, are
defined by
\begin{equation}  \label{qscs}
\begin{aligned}
Ric(A,B)=R(e_a,A,B,e_a),\quad S=\frac1{ 8n(n+2)}Ric(e_a,e_a),\quad
\rho_s(A,B)=\frac1{4n}R(A,B,e_a,I_se_a). 
\end{aligned}
\end{equation}

\begin{dfn}  A qc structure is said to be qc-Einstein if the horizontal qc-Ricci tensor is a scalar multiple of the metric,
$$Ric(X,Y)=2(n+2)Sg(X,Y).$$
\end{dfn}
The horizontal qc-Ricci tensor and the horizontal qc-Ricci 2-forms can be expressed in terms of the
torsion of the Biquard connection \cite{IMV} (see also
\cite{IV}). We collect the necessary facts from
\cite[Theorem~1.3, Theorem~3.12, Corollary~3.14, Proposition~4.3 and
Proposition~4.4]{IMV} with slight modification presented in
\cite{IV}.

\begin{thrm}\cite{IMV}\label{sixtyseven} On a $(4n+3)$-dimensional qc manifold $(M,\eta,\mathbb{Q})$ with a normalized qc scalar curvature $S$ we have the following relations
\begin{equation} \label{sixtyfour}
\begin{aligned}
& Ric(X,Y)  =(2n+2)T^0(X,Y)+(4n+10)U(X,Y)+2(n+2)Sg(X,Y),\\
& \rho_s(X,I_sY) = -\frac12\Bigl[T^0(X,Y)+T^0(I_sX,I_sY)\Bigr]-2U(X,Y)-Sg(X,Y),\\
& T(\xi_{i},\xi_{j}) =-S\xi_{k}-[\xi_{i},\xi_{j}]_{|H}, \qquad S  = -g(T(\xi_1,\xi_2),\xi_3).\\
 \end{aligned}
\end{equation}
For $n=1$ the above formulas hold with $U=0$.

\noindent The qc-Einstein condition
is equivalent to the vanishing of the torsion endomorphism of the
Biquard connection. In this case $S$ is constant and  the vertical distribution
is integrable (see \cite{IMV1} for $n=1$).
\end{thrm}
Any 3-Sasakian manifold has zero torsion endomorphism, and the
converse is true if in addition the qc-scalar  curvature  is a
positive constant \cite{IMV}.

The tensor $T^0$ determines the traceless $[-1]$-component of the horizontal qc-Ricci tensor, while the tensor $U$ determines the traceless part of the $[3]$-component of the horizontal qc-Ricci tensor \cite{IV2} (see also \cite{IMV,IV}):
\begin{equation}\label{Riccomp}
T^0=\frac{1}{2n+2}Ric_{[-1]},\qquad U=\frac{1}{4n+10}Ric_{[3][0]}.
\end{equation}
A weaker condition than the qc-Einstein one is contained in the following 
\begin{dfn}\cite[Definition~6.1]{IMV} Let $(M,g,\mathbb{Q})$  be a quaternionic contact manifold. We call M qc-pseudo-Einstein if the trace-free part of the
[3]-component of the qc-Ricci tensor vanishes, $U=0$. In dimension seven any qc manifold is qc-pseudo-Einstein.
\end{dfn}

We also give the following
\begin{dfn}
Let $(M,g,\mathbb{Q})$  be a quaternionic contact manifold. We call M qc-nearly-Einstein if the
[-1]-component of the qc-Ricci tensor vanishes, $T^0=0$.
\end{dfn}
We note that qc-nearly-Einstein manifolds are characterized by the condition that the
almost contact structure on the corresponding twistor space is normal, see \cite{DIM}.
Examples of qc-nearly-Einstein manifolds are provided by sub-Riemannian manifolds with transverse symmetry of qc type  as well as by Riemannian foliations with totally geodesic fibres of qc type (see \cite{BR1,BR2} and references therein).

We  use the contracted second Bianchi identity,  established in \cite{IMV},   in the form presented in \cite{IV,IV2}:
\begin{equation}\label{Bi2}
(n-1)(\bi_{e_a}T^0)(e_a,X)+2(n+2)(\bi_{e_a}U)(e_a,X)-(n-1)(2n+1)dS(X)=0.
\end{equation}
Clearly, the contracted second Bianchi identity \eqref{Bi2} shows that for $n>1$ any qc-Einstein space has constant qc scalar curvature but for qc-pseudo-Einstein and qc-nearly-Einstein spaces the qc scalar curvature could  not be a constant.

\subsection{The Ricci identities}

We  use repeatedly the Ricci identities of order two and three,
see also \cite{IV}. Let $f$ be a smooth function on the qc
manifold $M$ with horizontal gradient $\nabla f$ defined by
$g(\nabla f,X)=df(X)$. The sub-Laplacian of $f$ is, by definition, $ \triangle
f=-\sum_{a=1}^{4n}\nabla^2f(e_a,e_a)$. We have the following Ricci
identities (see e.g. \cite{IMV,IV2})
\begin{equation}  \label{rid}
\begin{aligned}&  \nabla^2f
(X,Y)-\nabla^2f(Y,X)=-2\sum_{s=1}^3\omega_s(X,Y)df(\xi_s), \\ &
\nabla^2f (X,\xi_s)-\nabla^2f(\xi_s,X)=T(\xi_s,X,\nabla f),\\ &
\nabla^3 f (X,Y,Z)-\nabla^3 f(Y,X,Z)=-R(X,Y,Z,\nabla f) -
2\sum_{s=1}^3 \omega_s(X,Y)\nabla^2f (\xi_s,Z).
\end{aligned}
\end{equation}
In view of \eqref{rid} we have the decompositions
\begin{equation}  \label{boh2}
\begin{aligned}  &(\bi^2f)_{[3][0]}(X,Y)=(\bi^2f)_{[3]}(X,Y) +\frac1{4n}\triangle fg(X,Y),\\&|(\nabla^2f)_{[3][0]}|^2=|(\nabla^2f)_{[3]}|^2-\frac{1}{4n}(\Delta f)^2,\\&(\bi^2f)_{[-1]}(X,Y)= (\nabla^2f)_{[-1][sym]}
(X,Y)+(\nabla^2f)_{[-1][a]}(X,Y)\\&\qquad\qquad\qquad\quad=(\nabla^2f)_{[-1][sym]}(X,Y)-\sum_{s=1}^3\omega_s(X,Y)df(\xi_s),\\ &|(\bi^2f)_{[-1]}|^2=|(\nabla^2f)_{[-1][sym]}|^2+4n\sum_{s=1}^3df(\xi_s)^2.
\end{aligned}
\end{equation}
We also need  the qc-Bochner formula \cite[(4.1)]{IPV1}
\begin{multline}  \label{bohS}
-\frac12\triangle |\nabla f|^2=|\nabla^2f|^2-g\left (\nabla
(\triangle f), \nabla f \right )+2(n+2)S|\nabla
f|^2+2(n+2)T^0(\nabla f,\nabla f) \\ +2(2n+2)U(\nabla f,\nabla f)+
4\sum_{s=1}^3\nabla^2f(\xi_s,I_s\nabla f). 
\end{multline}
Let $(M, g,\mathbb{Q})$ be a qc manifold of dimension $4n+3\geq 7$. For a
fixed local 1-form $\eta$ and a fixed $s\in \{1,2,3\}$ the form
$Vol_{\eta}=\eta_1\wedge\eta_2\wedge\eta_3\wedge\omega_s^{2n}$
is a locally defined volume form. Note that $Vol_{\eta}$ is independent of $%
s $ as well as the local one forms $\eta_1,\eta_2,\eta_3 $. Hence, it is a
globally defined volume form. The (horizontal) divergence of a horizontal
vector field/one-form $\sigma\in\Lambda^1\, (H)$, defined by $\nabla^*
\sigma =-tr|_{H}\nabla\sigma= -\nabla \sigma(e_a,e_a),$ 
 supplies the integration by parts formula \cite{IMV}, see also \cite{Wei},
\begin{equation*}  \label{div}
\int_M (\nabla^*\sigma)\,\, Vol_{\eta}\ =\ 0.
\end{equation*}

\subsection{The $P-$form}
We recall the definition of the P-form from \cite{IPV3}. Let
$(M,g,\mathbb{Q})$ be a compact quaternionic contact manifold of
dimension $4n+3$.

For a smooth function $f$ on $M$ the $P-$form $P_f $ on $M$ is defined  by \cite{IPV3}
\begin{equation}
\begin{aligned} P_f(X) =&\nabla ^{3}f(X,e_{a},e_{a})+\sum_{t=1}^{3}\nabla
^{3}f(I_{t}X,e_{a},I_{t}e_{a})-4nSdf(X)+4nT^{0}(X,\nabla f)\\ &\hskip2.8in
-\frac{8n(n-2)}{n-1}U(X,\nabla f), \quad \text{if $n>1$},\\ P_f(X) =&\nabla
^{3}f(X,e_{a},e_{a})+\sum_{t=1}^{3}\nabla
^{3}f(I_{t}X,e_{a},I_{t}e_{a})-4Sdf(X)+4T^{0}(X,\nabla f),\ \text{if $n=1$}.
\end{aligned}  \label{e:def P}
\end{equation}
We say that the $P-$function of $f$ is non-negative if the integral of $P_f(\gr)$  exists and is non-positive:
\begin{equation}  \label{e:non-negative Paneitz}
 -\int_M P_f(\nabla f)\, Vol_{\eta}\geq 0.
\end{equation}
If \eqref{e:non-negative Paneitz} holds for any smooth functon of
compact support we say that the $P$-function is non-negative. It
turns out that the $P$-function is non-negative on any compact qc
manifold of dimension at least eleven \cite{IPV3}. Indeed, \cite[Theorem~3.3]{IPV3} asserts that on a compact qc manifold af dimension bigger than seven, the next formula holds
\begin{equation}\label{p-form}
\int_MP_f(\nabla f)\,Vol_{\eta}=-\frac{4n}{n-1}\int_M|(\nabla^2f)_{[3][0]}|^2\, Vol_{\eta}.
\end{equation}




We  recall also the following integral identities,
\begin{multline}\label{iden1}
\int_M\sum_{s=1}^3\nabla^2f(\xi_s,I_s\nabla f)\, Vol_{\eta} \\
=\int_M\Big[\frac3{4n}|(\nabla^2f)_{[3]}|^2-\frac1{4n}|(\nabla^2f)_{[-1]}|^2-\frac{n+2}{2n}T^0(\gr,\gr) -\frac32S|\gr|^2\Big]\, Vol_{\eta},
\end{multline}
 proved in \cite[Lemma~3.3]{IPV1}, and 
\begin{multline}\label{new}
\int_M|(\bi^2f)_{[-1][a]}|^2\Vol=4n\int_M\sum_{s=1}^3(df(\xi_s))^2\Vol=\int_M\Big[ \frac1{4n}P_f(\gr) +\frac1{4n}(\triangle f)^2\\+S|\gr|^2-T^0(\gr,\gr)+\frac{2n-4}{n-1}U(\gr,\gr)\Big]\Vol,
\end{multline}
established in \cite[formula (4.11)]{IPV2}.

\begin{prop}For any smooth function $f$ on a compact qc manifold $(M,g,\mathbb{Q})$  of dimension 4n+3 bigger than seven we have
\begin{equation}\label{p-functf}
0=\int_M\Big[\frac{2n+1}{2n}(\triangle f)^2
-|(\nabla^2f)_{[-1][sym]}|^2-\frac{n+2}{n-1}|(\nabla^2f)_{[3][0]}|^2-(2n+1)Lic(\gr,\gr)\Big]\, Vol_{\eta},
\end{equation}
where
\begin{equation}\label{lli}Lic(\gr,\gr)=S|\gr|^2+\frac{2n+3}{2n+1}T^0(\gr,\gr)+\frac{2(n+2)(2n-1)}{(2n+1)(n-1)}U(\gr,\gr),
\end{equation}
and also
\begin{equation}\label{p-functC}
0=\int_M\Big[\frac{4n+3}{4n}(\triangle f)^2-|(\nabla^2f)_{[-1]}|^2 -\frac{n+3}{n-1}|(\nabla^2f)_{[3][0]}|^2-2nCor(\gr,\gr)\Big]\, Vol_{\eta},
\end{equation}
where
\begin{equation}\label{crr} Cor(\gr,\gr)= S|\gr|^2+\frac{n+2}{n}T^0(\gr,\gr)+\frac{2(n+1)}{n-1}U(\gr,\gr).
\end{equation}
\end{prop}
\begin{proof}
Integrating the qc Bochner formula \eqref{bohS} over the compact $M$ and applying \eqref{sixtyfour} 
 and \eqref{iden1}, we get
\begin{multline}\label{n1}
0=\int_M\Big[|\nabla^2f|^2 - (\triangle f)^2+Ric(\gr,\gr)+2T^0(\gr,\gr)\\-6U(\gr,\gr)+4\sum_{s=1}^3\nabla^2f(\xi_s,I_s\nabla f)\Big]\,Vol_{\eta}\\=
\int_M\Big[|\nabla^2f|^2 - (\triangle f)^2+Ric(\gr,\gr)+2T^0(\gr,\gr)-6U(\gr,\gr)\Big]\, Vol_{\eta}\\+\int_M\Big[\frac3{n}|(\nabla^2f)_{[3]}|^2-\frac1{n}|(\nabla^2f)_{[-1]}|^2-\frac{2(n+2)}{n}T^0(\gr,\gr) -6S|\gr|^2\Big]\, Vol_{\eta}\\
=\int_M\Big[|(\nabla^2f)_{[3][0]}|^2+\frac1{4n}(\triangle f)^2+|(\nabla^2f)_{[-1]}|^2 - (\triangle f)^2+\frac3{n}|(\nabla^2f)_{[3][0]}|^2+\frac{3}{4n^2}(\triangle f)^2-\frac1{n}|(\nabla^2f)_{[-1]}|^2\Big]\, Vol_{\eta}\\+\int_M\Big[Ric(\gr,\gr)-\frac4nT^0(\gr,\gr)-6U(\gr,\gr)\ -6S|\gr|^2\Big]\, Vol_{\eta}\\
=\int_M\Big[\frac{n+3}{n}|(\nabla^2f)_{[3][0]}|^2-\frac{(n-1)(4n+3)}{4n^2}(\triangle f)^2+\frac{n-1}{n}|(\nabla^2f)_{[-1]}|^2 +2(n-1)Cor(\gr,\gr)\Big]\, \Vol,
\end{multline}
which proves \eqref{p-functC}.

The substitution of the fourth equality of  \eqref{boh2} and \eqref{new} into \eqref{p-functC} leads to
\begin{multline*}
0=\int_M\Big[\frac{4n+3}{4n}(\triangle f)^2-|(\nabla^2f)_{[-1]}|^2 -\frac{n+3}{n-1}|(\nabla^2f)_{[3][0]}|^2-2nCor(\gr,\gr)\Big]\, Vol_{\eta}\\=\int_M\Big\{\frac{4n+3}{4n}(\triangle f)^2 -\frac{n+3}{n-1}|(\nabla^2f)_{[3][0]}|^2
-|(\nabla^2f)_{[-1][sym]}|^2-\frac1{4n}\Big[P_f(\gr)+(\triangle f)^2\Big]\\-\Big[S|\gr|^2-T^0(\gr,\gr)+\frac{2n-4}{n-1}U(\gr,\gr)\ \Big]\\
-2n\Big[\frac{n+2}{n}T^0(\gr,\gr)+\frac{2(n+1)}{n-1}U(\gr,\gr)+S|\gr|^2\Big] \Big\}\, Vol_{\eta},
\end{multline*}
which proves
 \eqref{p-functf} in view of \eqref{p-form}.
\end{proof}
\begin{rmrk}
Note that the expression in the formula \eqref{lli}	appears in the qc Lichnerowicz-type positivity condition used to find a lower bound of the first eigenvalue of the sub-Laplacian  (see \cite[formula (1.1)]{IPV1} and \cite[formula (1.2)]{IPV3}). On the other  hand,  the expression in \eqref{crr} appears in the qc Cordes-type a
priori inequality between the (horizontal) Hessian and the sub-Laplacian of a function, derived in \cite[formula (1.2)]{IPV1}.
\end{rmrk}
Indeed, in view of \eqref{boh2}, the equality \eqref{p-functC} can be written in the form
\begin{multline}\label{p-functC1}
\frac{n+1}n\int_M(\triangle f)^2\Vol=\int_M|\bi f|^2\Vol+\int_M\Big[\frac4{n-1}|(\nabla^2f)_{[3][0]}|^2+2nCor(\gr,\gr)\Big]\Vol.
\end{multline}
In this way we recovered the next result from \cite{IPV1}, supplying more information in the equality case:
\begin{thrm}\cite{IPV1}
On a compact qc manifold of dimension bigger than seven one has the inequality
\begin{equation}\label{ccc}
\int_M(\triangle f)^2\Vol\ge \frac{n}{n+1}\int_M|\bi f|^2\Vol+\frac{2n^2}{n+1}\int_M Cor(\gr,\gr)\Vol.
\end{equation}
The equality in \eqref{ccc} can be achieved only for functions with vanishing trace-free part of the $[3]$-component of the horizontal Hessian.
\end{thrm}

\section{The inequalities, main results}
Denote by $\bar S$ the  average value of the scalar curvature,  $$\bar S=\int_MS\,Vol_{\eta}.$$
For a smooth function $f$ the sub-Laplacian $\triangle f$ is a subelliptic operator. According to \cite{Eg} and \cite{hor}, let  $\varphi$ be the unique solution of the following PDE:
\begin{equation}\label{scal}
\triangle\varphi=S-\bar{S}, \quad \int_M\varphi\,Vol_{\eta}=0.
\end{equation}
We have
\begin{thrm}\label{thrm1}
Let $(M,g,\mathbb{Q})$ be a compact qc manifold of dimension $(4n+3),\quad n>1$.
\begin{itemize}
\item[a)] Suppose  the next positivity condition is satisfied
\begin{equation}\label{posl1}
Lic(X,X)=Sg(X,X)+\frac{2n+3}{2n+1}T^0(X,X)+\frac{2(n+2)(2n-1)}{(n-1)(2n+1)}U(X,X)\ge 0, \quad X\in H.
\end{equation}
Then we have
\begin{multline}\label{inl7}
\int_M(S-\bar S)^2\Vol\le
\frac{n+2}{2n(n-1)(2n+5)^2(2n+1)}\int_M|Ric_{[3][0]}|^2\Vol\\-\frac 2{2n+1}\int_MT^0(e_a,e_b)(\bi^2\varphi)_{[-1][sym]}(e_a,e_b)\Vol
\\=\frac{2(n+2)}{n(n-1)(2n+1)}\int_M|U|^2\Vol-\frac 2{2n+1}\int_MT^0(e_a,e_b)(\bi^2\varphi)_{[-1][sym]}(e_a,e_b)\Vol.
\end{multline}
If the equality in \eqref{inl7} holds then
\begin{multline}\label{inlf7}
\int_M(S-\bar S)^2\Vol=
\frac{n+2}{2n(n-1)(2n+5)^2(2n+1)}\int_M|Ric_{[3][0]}|^2\Vol\\
=\frac{2(n+2)}{n(n-1)(2n+1)}\int_M|U|^2\Vol
\end{multline}
and the qc conformal structure $\bar\eta=\frac{n}{2}e^{-2\varphi}\eta$ will be qc-pseudo-Einstein.
\item[b)] Suppose  the next positivity condition is satisfied
\begin{equation}\label{ccor}
Cor(X,X)= Sg(X,X)+\frac{n+2}{n}T^0(X,X)+\frac{2(n+1)}{n-1}U(X,X)\ge 0, \quad X\in H.
\end{equation}
Then we have
\begin{multline}\label{cinl7}
\int_M(S-\bar S)^2\Vol\le
\frac{(n+2)^2(4n+3)}{4n(n-1)(n+3)(2n+5)^2(2n+1)^2}\int_M|Ric_{[3][0]}|^2\Vol\\-\frac 2{2n+1}\int_MT^0(e_a,e_b)(\bi^2\varphi)_{[-1][sym]}(e_a,e_b)\Vol
\\=\frac{(n+2)^2(4n+3)}{n(n-1)(n+3)(2n+1)^2}\int_M|U|^2\Vol-\frac 2{2n+1}\int_MT^0(e_a,e_b)(\bi^2\varphi)_{[-1][sym]}(e_a,e_b)\Vol.
\end{multline}
If the equality in \eqref{cinl7} holds then
\begin{multline}\label{cinlf7}
\int_M(S-\bar S)^2\Vol=
\frac{(n+2)^2(4n+3)}{4n(n-1)(n+3)(2n+5)^2(2n+1)^2}\int_M|Ric_{[3][0]}|^2\Vol
\\=\frac{(n+2)^2(4n+3)}{n(n-1)(n+3)(2n+1)^2}\int_M|U|^2\Vol
\end{multline}
and the qc conformal structure $\bar\eta=\frac{n(n+3)(2n+1)}{(n+2)(4n+3)}e^{-2\varphi}\eta$ will be qc-pseudo-Einstein.
\end{itemize}
\end{thrm}
\begin{proof}  The proof follows the approach of \cite{LT}.
We have
\begin{equation}\label{in1}
\int_M(S-\bar S)^2\Vol=\int_M(S-\bar S)\triangle\varphi\Vol=\int_MdS(\g)\Vol,
\end{equation}
where we used \eqref{scal} and an integration by parts to achieve the last equality.

The equalities \eqref{Riccomp} and \eqref{Bi2} imply
\begin{multline}\label{ric0}
dS(X)=\frac{1}{2n+1}(\bi_{e_a}T^0)(e_a,X)+\frac{2n+4}{(n-1)(2n+1)}(\bi_{e_a}U)(e_a,X)\\=
\frac{n+2}{(n-1)(2n+5)(2n+1)}(\bi_{e_a}Ric_{[3][0]})(e_a,X) +\frac{1}{2n+1}(\bi_{e_a}T^0)(e_a,X)\\=
\frac{1}{2(n+1)(2n+1)}(\bi_{e_a}Ric_{[-1]})(e_a,X)+\frac{2n+4}{(n-1)(2n+1)}(\bi_{e_a}U)(e_a,X).
\end{multline}
Using \eqref{ric0}, we obtain from \eqref{in1} and an integration by parts that
\begin{multline}\label{in2}
\int_M(S-\bar S)^2\Vol=\int_MdS(\g)\Vol\\=
\frac{n+2}{(n-1)(2n+5)(2n+1)}\int_M(\bi_{e_a}Ric_{[3][0]})(e_a,\g)\Vol +\frac{1}{2n+1}\int_M(\bi_{e_a}T^0)(e_a,\g)\Vol\\
=-\frac{n+2}{(n-1)(2n+5)(2n+1)}\int_M(Ric_{[3][0]})(e_a,e_b)(\bi^2\varphi)_{[3][0]}(e_a,e_b)\Vol\\ +\frac 1{2n+1}\int_M(\bi_{e_a}T^0)(e_a,\g)\Vol.
\end{multline}
Applying the Young's inequality 
\begin{equation}\label{yu} 2pq\le dp^2+d^{-1} q^2,
\end{equation} for a constant $d>0$, we get from \eqref{in2}  that
\begin{multline}\label{in3}
\int_M(S-\bar S)^2\Vol\le
\frac{n+2}{2(n-1)(2n+5)(2n+1)}\int_M\Big[d|Ric_{[3][0]}|^2+d^{-1}|(\bi^2\varphi)_{[3][0]}|^2\Big]\Vol\\ +\frac 1{2n+1}\int_M(\bi_{e_a}T^0)(e_a,\g)\Vol.
\end{multline}
First we prove a). The equalities \eqref{p-functf}, \eqref{scal} and the positivity condition \eqref{posl1} help us to obtain from \eqref{in3} the next inequality
\begin{multline}\label{inl4}
\int_M(S-\bar S)^2\Vol\le
\frac{d(n+2)}{2(n-1)(2n+5)(2n+1)}\int_M|Ric_{[3][0]}|^2\Vol\\+\frac{1}{4nd(2n+5)}\int_M(S-\bar S)^2\Vol +\frac 1{2n+1}\int_M(\bi_{e_a}T^0)(e_a,\g)\Vol.
\end{multline}
Thus,  \eqref{inl4} yields
\begin{multline}\label{inl5}
\Big[1-\frac{1}{4nd(2n+5)} \Big]\int_M(S-\bar S)^2\Vol
\le
\frac{d(n+2)}{2(n-1)(2n+5)(2n+1)}\int_M|Ric_{[3][0]}|^2\Vol\\+\frac 1{2n+1}\int_M(\bi_{e_a}T^0)(e_a,\g)\Vol.
\end{multline}
Set 
\begin{equation}\label{d1}
d=\frac{1}{4n^2+10n}
\end{equation}
 into \eqref{inl5} to get \eqref{inl7}, which proves the first part of a).

If we have an equality in \eqref{inl7}, we put the expression of $|(\bi^2f)_{[3][0]}|^2$ from \eqref{p-functf} into \eqref{in3}, taken with the constant  $d$ given by \eqref{d1}, to get
\begin{equation*}
0\le -\int_M\Big[|(\nabla^2\varphi)_{[-1][sym]}|^2+(2n+1)Lic(\g,\g)\Big]\Vol,
\end{equation*}
which, in view of \eqref{posl1}, yields
\begin{equation}\label{eqll2}
\begin{aligned}
&(\nabla^2\varphi)_{[-1][sym]}=0;\\
&Lic(\g,\g)=S|\g|^2+\frac{2n+3}{2n+1}T^0(\g,\g)+\frac{2(n+2)(2n-1)}{(n-1)(2n+1)}U(\g,\g)=0.
\end{aligned}
\end{equation}
The first equality in \eqref{eqll2} and the equality in \eqref{inl7} give \eqref{inlf7}.
Now, the equalities \eqref{inlf7}, \eqref{p-functf} and \eqref{eqll2}, together with \eqref{in1}, imply
\begin{equation}\label{eql3}
\begin{aligned}
&\int_M(S-\bar S)^2\Vol=
\frac{n+2}{2n(n-1)(2n+5)^2(2n+1)}\int_M|Ric_{[3][0]}|^2\Vol;\\
&\int_M|(\bi^2\varphi)_{[3][0]}|^2\Vol=\frac{(2n+1)(n-1)}{2n(n+2)}\int_M(S-\bar S)^2\Vol=\frac{1}{4n^2(2n+5)^2}\int_M|Ric_{[3][0]}|^2\Vol.
\end{aligned}
\end{equation}
On the other hand, \eqref{in2} and \eqref{eql3} yield
\begin{multline}\label{eql4}
\frac{n+2}{2n(n-1)(2n+5)^2(2n+1)}\int_M|Ric_{[3][0]}|^2\Vol\\=
-\frac{n+2}{(n-1)(2n+5)(2n+1)}\int_M(Ric_{[3][0]})(e_a,e_b)(\bi^2\varphi)_{[3][0]}(e_a,e_b)\Vol.
\end{multline}
 The equalities \eqref{eql3} and \eqref{eql4}  imply
 \begin{equation}\label{eqll5}
 (\bi^2\varphi)_{[3][0]}=-\frac{1}{2n(2n+5)}Ric_{[3][0]}=-\frac{1}{n}U.
 \end{equation}
Taking the qc conformal change $\bar\eta=\frac1{4d(2n+5)}e^{-2\varphi}\eta=\frac{n}{2}e^{-2\varphi}\eta$, we get $\bar U=0$, according to \cite[formula (5.6)]{IMV}, taken with $2h=4d(2n+5)e^{2\varphi}=\frac{2}{n}e^{2\varphi}$ (cf. also \cite[formula (5.23)]{IV2}),  and \eqref{eqll5}. Hence, $\overline{Ric}_{[3][0]}=0$ and the qc conformal qc structure $\bar\eta$ is qc-pseudo-Einstein, which completes the proof of a).
 
Now we prove b).  The equalities \eqref{p-functC}, \eqref{scal} and the positivity condition \eqref{ccor} help us to obtain from \eqref{in3} the following inequality
\begin{multline}\label{cinl4}
\int_M(S-\bar S)^2\Vol\le
\frac{d(n+2)}{2(n-1)(2n+5)(2n+1)}\int_M|Ric_{[3][0]}|^2\Vol\\+\frac{(n+2)(4n+3)}{8nd(2n+5)(2n+1)(n+3)}\int_M(S-\bar S)^2\Vol +\frac 1{2n+1}\int_M(\bi_{e_a}T^0)(e_a,\g)\Vol.
\end{multline}
Thus,  \eqref{cinl4} yields
\begin{multline}\label{cinl5}
\Big[1-\frac{(n+2)(4n+3)}{8nd(2n+5)(2n+1)(n+3)} \Big]\int_M(S-\bar S)^2\Vol
\\\le
\frac{d(n+2)}{2(n-1)(2n+5)(2n+1)}\int_M|Ric_{[3][0]}|^2\Vol+\frac 1{2n+1}\int_M(\bi_{e_a}T^0)(e_a,\g)\Vol.
\end{multline}
Set 
\begin{equation}\label{d2}
d=\frac{(n+2)(4n+3)}{4n(2n+5)(2n+1)(n+3)}
\end{equation}
into \eqref{cinl5} to get \eqref{cinl7}, which proves the first part of b).

If we have an equality in \eqref{cinl7}, we put the expression of $|(\bi^2f)_{[3][0]}|^2$ from \eqref{p-functC} into \eqref{in3}, taken with the constant $d$ given by \eqref{d2},  to get
\begin{equation*}
0\le -\int_M\Big[|(\nabla^2\varphi)_{[-1]}|^2+2n Cor(\g,\g)\Big]\Vol,
\end{equation*}
which, in view of \eqref{ccor}, yields 
\begin{equation}\label{ceqll2}
\begin{aligned}
&(\bi^2\varphi)_{[-1]}=0; \\ &Cor(\g,\g)= Sg(\g,\g)+\frac{n+2}{n}T^0(\g,\g)+\frac{2(n+1)}{n-1}U(\g,\g)=0.
\end{aligned}
\end{equation}
The first equality in \eqref{ceqll2} and the equality in \eqref{cinl7} imply \eqref{cinlf7}.
Now, the equalities \eqref{cinlf7}, \eqref{p-functC} and \eqref{ceqll2} together with \eqref{in1} imply
\begin{equation}\label{ceql3}
\begin{aligned}
&\int_M(S-\bar S)^2\Vol=
\frac{(n+2)^2(4n+3)}{4n(n-1)(n+3)(2n+5)^2(2n+1)^2}\int_M|Ric_{[3][0]}|^2\Vol;\\
&\int_M|\bi^2\varphi_{[3][0]}|^2\Vol=\frac{(4n+3)(n-1)}{4n(n+3)}\int_M(S-\bar S)^2\Vol\\&\qquad\qquad\qquad\qquad\quad=\frac{(n+2)^2(4n+3)^2}{16n^2(n+3)^2(2n+5)^2(2n+1)^2}\int_M|Ric_{[3][0]}|^2\Vol.
\end{aligned}
\end{equation}
On the other hand, \eqref{in2} and \eqref{ceql3} yield
\begin{multline}\label{ceql4}
\frac{(n+2)^2(4n+3)}{4n(n-1)(n+3)(2n+5)^2(2n+1)^2}\int_M|Ric_{[3][0]}|^2\Vol\\=
-\frac{n+2}{(n-1)(2n+5)(2n+1)}\int_M(Ric_{[3][0]})(e_a,e_b)(\bi^2\varphi)_{[3][0]}(e_a,e_b)\Vol.
\end{multline}
 The equalities \eqref{ceql3} and \eqref{ceql4}  imply
 \begin{equation}\label{ceqll5}
 (\bi^2\varphi)_{[3][0]}=-\frac{(n+2)(4n+3)}{4n(n+3)(2n+5)(2n+1)}Ric_{[3][0]}=-\frac{(n+2)(4n+3)}{2n(n+3)(2n+1)}U.
 \end{equation} 
 Take the qc conformal change $\bar\eta=\frac1{4d(2n+5)}e^{-2\varphi}\eta=\frac{n(n+3)(2n+1)}{(n+2)(4n+3)}e^{-2\varphi}\eta$ gives $\bar U=0$, according to \cite[formula (5.6)]{IMV}, taken with $2h=4d(2n+5)e^{2\varphi}=\frac{(n+2)(4n+3)}{n(n+3)(2n+1)}e^{2\varphi}$ (cf. also \cite[formula (5.23)]{IV2}),  and \eqref{ceqll5}. Hence, $\overline{Ric}_{[3][0]}=0$ and the qc conformal qc structure $\bar\eta$ is qc-pseudo-Einstein, which completes the proof of b).
\end{proof}
An integration by parts of the last term in \eqref{inl7} yields the following
\begin{cor}\label{coe1}
In addition to the same conditions as in Theorem~\ref{thrm1}, we assume
\begin{equation}\label{to}
(\bi_{e_a}\bi_{e_b}T^0)(e_a,e_b)=0.
\end{equation}
\begin{itemize}
\item[a)] If $Lic(X,X)\ge 0$ then
\begin{equation*}
\int_M(S-\bar S)^2\Vol\le
\frac{n+2}{2n(n-1)(2n+5)^2(2n+1)}\int_M|Ric_{[3][0]}|^2\Vol
=\frac{2(n+2)}{n(n-1)(2n+1)}\int_M|U|^2\Vol.
\end{equation*}
\item[b)] If $Cor(X,X)\ge 0$ then
\begin{multline*}
\int_M(S-\bar S)^2\Vol\le
\frac{(n+2)^2(4n+3)}{4n(n-1)(n+3)(2n+5)^2(2n+1)^2}\int_M|Ric_{[3][0]}|^2\Vol
\\=\frac{(n+2)^2(4n+3)}{n(n-1)(n+3)(2n+1)^2}\int_M|U|^2\Vol.
\end{multline*}
\end{itemize}
In both cases, if M is qc-pseudo-Einstein then the qc scalar curvature $S$ is a constant.
\end{cor}

Further, in a  similar way as the proof of Theorem~\ref{thrm1}, we obtain
\begin{thrm}\label{thrm2}
Let $(M,g,\mathbb{Q})$ be a compact qc manifold of dimension $(4n+3),\quad n>1$.
\begin{itemize}
\item[a)] Suppose  the  positivity condition \eqref{posl1} holds. Then we have
\begin{multline}\label{inl7n}
\int_M(S-\bar S)^2\Vol\le
\frac{1}{8n(n+1)^2(2n+1)}\int_M|Ric_{[-1]}|^2\Vol\\+\frac{4(n+2)}{(n-1)(2n+1)}\int_M(\bi_{e_a}U)(e_a,\g)\Vol
\\=\frac{1}{2n(2n+1)}\int_M|T^0|^2\Vol-\frac{4(n+2)}{(n-1)(2n+1)}\int_MU(e_a,e_b)(\bi^2\varphi)_{[3][0]}(e_a,e_b)\Vol.
\end{multline}
If the equality holds then
\begin{equation}\label{inlf7n}
\int_M(S-\bar S)^2\Vol=
\frac{1}{8n(n+1)^2(2n+1)}\int_M|Ric_{[-1]}|^2\Vol
=\frac{1}{2n(2n+1)}\int_M|T^0|^2\Vol.
\end{equation}
\item[b)] Suppose  the  positivity condition \eqref{ccor} holds. Then we have
\begin{multline}\label{inl7nc}
\int_M(S-\bar S)^2\Vol\le\frac{4n+3}{16n(n+1)^2(2n+1)^2}\int_M|Ric_{[-1]}|^2\Vol\\+\frac{4(n+2)}{(n-1)(2n+1)}\int_M(\bi_{e_a}U)(e_a,\g)\Vol
\\=\frac{4n+3}{4n(2n+1)^2}\int_M|T^0|^2\Vol-\frac{4(n+2)}{(n-1)(2n+1)}\int_MU(e_a,e_b)(\bi^2\varphi)_{[3][0]}(e_a,e_b)\Vol.
\end{multline}
If the equality holds then
\begin{equation}\label{inlf7nc}
\int_M(S-\bar S)^2\Vol=
\frac{4n+3}{16n(n+1)^2(2n+1)^2}\int_M|Ric_{[-1]}|^2\Vol
=\frac{4n+3}{4n(2n+1)^2}\int_M|T^0|^2\Vol.
\end{equation}
\end{itemize}
\end{thrm}
\begin{proof}
Using \eqref{ric0}, we obtain from \eqref{in1} and an integration by parts that
\begin{multline}\label{in2n}
\int_M(S-\bar S)^2\Vol=\int_MdS(\g)\Vol\\=
\frac{1}{2(n+1)(2n+1)}\int_M\Big[(\bi_{e_a}Ric_{[-1]})(e_a,\g) +\frac{4(n+1)(n+2)}{n-1}(\bi_{e_a}U)(e_a,\g) \Big]\Vol\\=
-\frac{1}{2(n+1)(2n+1)}\int_M(Ric_{[-1]})(e_a,e_b)(\bi^2\varphi)_{[-1][sym]}(e_a,e_b) \Vol\\+\frac{2(n+2)}{(n-1)(2n+1)}\int_M(\bi_{e_a}U)(e_a,\g)\Vol.
\end{multline}
First we prove a).
Applying  the Young's inequality \eqref{yu} for a constant $d>0$, together with \eqref{p-functf}, we get from \eqref{in2n} 
\begin{multline}\label{in3n}
\int_M(S-\bar S)^2\Vol\le
\frac{1}{4(n+1)(2n+1)}\Big[d\int_M|Ric_{[-1]}|^2\Vol+d^{-1}\int_M|(\bi^2\varphi)_{[-1][sym]}|^2\Vol\Big]\\ +\frac{2(n+2)}{(n-1)(2n+1)}\int_M(\bi_{e_a}U)(e_a,\g)\Vol\\=\frac{d}{4(n+1)(2n+1)}\int_M|Ric_{[-1]}|^2\Vol+\frac{2(n+2)}{(n-1)(2n+1)}\int_M(\bi_{e_a}U)(e_a,\g)\Vol\\+\frac1{2d}\int_M\Big[\frac{1}{4n(n+1)}(\triangle\varphi)^2 -\frac{n+2}{2(n^2-1)(2n+1)}|(\bi^2\varphi)_{[3][0]}|^2-\frac1{2(n+1)}Lic(\g,\g)\Big]\Vol.
\end{multline}
The equalities \eqref{scal}, the positivity condition \eqref{posl1} and  \eqref{in3n} yield
\begin{multline}\label{inl4n}
\int_M(S-\bar S)^2\Vol\le
\frac{d}{4(n+1)(2n+1)}\int_M|Ric_{[-1]}|^2\Vol\\+\frac{1}{8nd(n+1)}\int_M(S-\bar S)^2\Vol +\frac{2(n+2)}{(n-1)(2n+1)}\int_M(\bi_{e_a}U)(e_a,\g)\Vol.
\end{multline}
Thus,  \eqref{inl4n} yields
\begin{multline}\label{inl5n}
\Big[1-\frac{1}{8nd(n+1)} \Big]\int_M(S-\bar S)^2\Vol\le
\frac{d}{4(n+1)(2n+1)}\int_M|Ric_{[-1]}|^2\Vol\\+\frac{2(n+2)}{(n-1)(2n+1)}\int_M(\bi_{e_a}U)(e_a,\g)\Vol.
\end{multline}
Set $d=\frac{1}{4n(n+1)}$ into \eqref{inl5n} to get \eqref{inl7n}.

If we have an equality in \eqref{inl7n}, we use  \eqref{p-functf} to get from \eqref{in3n}   that
\begin{equation*}
\begin{split}
0\le
-\int_M\Big[\frac{n+2}{n-1}|(\nabla^2\varphi)_{[3][0]}|^2+(2n+1)Lic(\g,\g)\Big]\Vol,
\end{split}
\end{equation*}
which, in view of \eqref{posl1}, yields
\begin{equation}\label{eql2}
\begin{split}
(\nabla^2\varphi)_{[3][0]}=0;\quad
S|\nabla\varphi|^2+\frac{2n+3}{2n+1}T^0(\nabla\varphi,\nabla\varphi)+\frac{2(n+2)(2n-1)}{(n-1)(2n+1)}U(\nabla\varphi,\nabla\varphi)=0.
\end{split}
\end{equation}
Now, the equalities \eqref{inl7n} 
 and \eqref{eql2} 
  imply \eqref{inlf7n}, which proves a).
  
 For  b),  using   the Young's inequality \eqref{yu} for a constant $d>0$ together with \eqref{p-functC}, we get from \eqref{in2n} 
\begin{multline}\label{in3nc}
\int_M(S-\bar S)^2\Vol\le
\frac{1}{4(n+1)(2n+1)}\Big[d\int_M|Ric_{[-1]}|^2\Vol+d^{-1}\int_M|(\bi^2\varphi)_{[-1]}|^2\Vol\Big]\\ +\frac{2(n+2)}{(n-1)(2n+1)}\int_M(\bi_{e_a}U)(e_a,\g)\Vol\\=\frac{d}{4(n+1)(2n+1)}\int_M|Ric_{[-1]}|^2\Vol+\frac{2(n+2)}{(n-1)(2n+1)}\int_M(\bi_{e_a}U)(e_a,\g)\Vol\\+\frac{1}{4d(n+1)(2n+1)}\int_M\Big[\frac{4n+3}{4n}(\triangle\varphi)^2 -\frac{n+3}{n-1}|(\bi^2\varphi)_{[3][0]}|^2-2nCor(\g,\g)\Big]\Vol.
\end{multline}
The equality \eqref{scal},  the positivity condition \eqref{ccor} and \eqref{in3nc} imply
\begin{multline}\label{inl5nc}
\Big[1-\frac{4n+3}{16nd(n+1)(2n+1)} \Big]\int_M(S-\bar S)^2\Vol\le
\frac{d}{4(n+1)(2n+1)}\int_M|Ric_{[-1]}|^2\Vol\\+\frac{2(n+2)}{(n-1)(2n+1)}\int_M(\bi_{e_a}U)(e_a,\g)\Vol.
\end{multline}
Set $d=\frac{4n+3}{8n(n+1)(2n+1)}$ into \eqref{inl5nc} to get \eqref{inl7nc}.

If we have an equality in \eqref{inl7nc}, we use  \eqref{p-functC} to get from \eqref{in3nc}   that
\begin{equation*}
\begin{split}
0\le
-\int_M\Big[\frac{n+3}{n-1}|(\nabla^2\varphi)_{[3][0]}|^2+2nCor(\g,\g)\Big]\Vol,
\end{split}
\end{equation*}
which, in view of \eqref{ccor}, yields
\begin{equation}\label{eql2c}
\begin{split}
(\nabla^2\varphi)_{[3][0]}=0;\quad
Cor(\g,\g)= Sg(\g,\g)+\frac{n+2}{n}T^0(\g,\g)+\frac{2(n+1)}{n-1}U(\g,\g)=0.
\end{split}
\end{equation}
Now, the equalities \eqref{inl7nc} 
 and \eqref{eql2c} 
  imply \eqref{inlf7nc}, which completes the proof.
\end{proof}
An integration by parts of the last term in \eqref{inl7n} and in \eqref{inl7nc}, respectively, yields the following
\begin{cor}\label{coe1c}
In addition to the same conditions as in Theorem~\ref{thrm2}, we assume
\begin{equation*}\label{u}
(\bi_{e_a}\bi_{e_b}U)(e_a,e_b)=0.
\end{equation*}
\begin{itemize}
\item[a)] If $Lic(X,X)\ge 0$ then
\begin{equation*}
\int_M(S-\bar S)^2\Vol\le
\frac{1}{8n(n+1)^2(2n+1)}\int_M|Ric_{[-1]}|^2\Vol
=\frac{1}{2n(2n+1)}\int_M|T^0|^2\Vol.
\end{equation*}
\item[b)] If $Cor(X,X)\ge 0$ then
\begin{equation*}
\int_M(S-\bar S)^2\Vol\le
\frac{4n+3}{16n(n+1)^2(2n+1)^2}\int_M|Ric_{[-1]}|^2\Vol
=\frac{4n+3}{4n(2n+1)^2}\int_M|T^0|^2\Vol.
\end{equation*}
\end{itemize}
In both cases, if M is qc-nearly-Einstein then the qc scalar curvature $S$ is a constant.
\end{cor}

\end{document}